\documentclass[12pt,twoside]{amsart}

\usepackage{typearea,hyperref,color}
\usepackage{amsmath}
\usepackage{amssymb}
\usepackage{amscd}
\usepackage{graphics}
\usepackage{graphicx}
\usepackage{epstopdf}
\usepackage{epsfig}

\renewcommand{\bar}{\overline}

\newcommand{\AAA}{\mathbb{A}}

\newcommand{\CC}{\mathbb{C}}

\newcommand{\PP}{\mathbb{P}}
\newcommand{\QQ}{\mathbb{Q}}
\newcommand{\RR}{\mathbb{R}}

\newcommand{\ZZ}{\mathbb{Z}}

\newcommand{\Qp}{\QQ_p}

\newcommand{\Cv}{\CC_v}

\newcommand{\calM}{{\mathcal M}}

\newcommand{\calP}{{\mathcal P}}

\newcommand{\kbar}{\overline{k}}

\newcommand{\PC}{\PP^1(\CC)}
\newcommand{\PCv}{\PP^1(\Cv)}

\newcommand{\Pkbar}{\PP^1(\overline{k})}

\newcommand{\Ber}{\textup{an}}

\newcommand{\PBerkv}{\PP^1_{\Ber,v}}

\DeclareMathOperator{\Rat}{Rat}
\DeclareMathOperator{\PGL}{PGL}

\DeclareMathOperator{\Gal}{Gal}

\newcommand{\Dbar}{\bar{D}}

\theoremstyle{plain}
\newtheorem{thm}{Theorem}[section]

\newtheorem{lemma}[thm]{Lemma}
\newtheorem{conj}[thm]{Conjecture}

\theoremstyle{definition}
\newtheorem{defin}[thm]{Definition}

\newtheorem{remark}[thm]{Remark}

\theoremstyle{remark}

\numberwithin{equation}{section}

\title[PCF unicritical polynomials]
{A finiteness property of postcritically finite unicritical polynomials}
\date{October 29, 2020}
\author{Robert~L. Benedetto}
\address[Benedetto]{Amherst College \\ Amherst, MA 01002}
\email{rlbenedetto@amherst.edu}
\author{Su-Ion Ih}
\address[Ih]{University of Colorado \\ Boulder, CO 80309 
and Korea Institute for Advanced Study, Seoul 02455}
\email{ih@math.colorado.edu}

\begin{document}

\begin{abstract}
Let $k$ be a number field with algebraic closure $\overline k$,
and let $S$ be a finite set of places of $k$
containing all the archimedean ones.
Fix $d\geq 2$ and $\alpha \in \overline k$ such that the map
$z\mapsto z^d+\alpha$ is not postcritically finite.
Assuming a technical hypothesis on $\alpha$, we prove that
there are only finitely many parameters $c\in\overline{k}$ for which
$z\mapsto z^d+c$ is postcritically finite
and for which $c$ is $S$-integral relative to $(\alpha)$.
That is,  in the moduli space of unicritical polynomials of degree $d$,
there are only finitely many PCF $\kbar$-rational points that are $((\alpha),S)$-integral.
We conjecture that the same statement is true without the technical hypothesis.
\end{abstract}



\maketitle


\hfill \emph{In memory of Lucien Szpiro}
\hfill
{}


\section{Introduction} \label{sec:intro}
Let $k$ be a field with algebraic closure $\kbar$, and let $f\in k(z)$ 
be a rational function defined over $k$.
A point $x\in\Pkbar$ is \emph{preperiodic} if $f^n(x)=f^m(x)$ for some 
integers $n>m\geq 0$,
where $f^n:=f\circ \cdots\circ f$ denotes the $n$-fold composition of $f$ with itself,
with $f^{0}:=\textup{id}$.
The map $f$ is said to be \emph{postcritically finite}, or PCF, if all of its critical points in $\Pkbar$
are preperiodic under the iteration of $f$.
In both complex and arithmetic dynamics,
PCF maps have proven themselves to be objects of particular interest
for their special dynamical and arithmetic properties.
See, for example,
\cite{AHM,ABEGKM,BD;13,BFHJY,BIJL,Buff;18,Eps;12,
FG;18,FavGau;arxiv,GKNY;17,HT;15,Jon;13,Koch;13}.
In particular, in an algebraic moduli space of discrete dynamical systems,
the points corresponding to PCF maps appear to play a similar role
as other special points, such as CM points on classical modular curves.
In this paper, for $d\geq 2$ an integer,
we consider PCF parameters in the one-parameter family
of unicritical polynomials $f_{d,c}(z):= z^d+c$, and we prove a finiteness
result concerning integrality of such parameters with respect to a given
non-PCF parameter.

The polynomial $f_{d,c}$ has critical points at $z=0,\infty$.
Since $\infty$ is fixed, it follows that
$f_{d,c}$ is PCF if and only if the forward orbit 
\[ \{ f_{d,c}^n(0) : n\geq 0 \} \]
of the critical point $z=0$ is a finite set.
Any such PCF parameter $c$ must lie in $\overline{\QQ}$,
since such $c$ is a root of the polynomial $f_{d,c}^n(0)-f_{d,c}^m(0)$
for some integers $n>m\geq 0$.
(In fact, such $c$ must be an algebraic integer, since this polynomial
is monic with integer coefficients.)
Moreover, by \cite[Theorem~1.1]{BIJL}, the PCF parameters form
a set of bounded arithmetic height. In particular, for any number field $k$,
there are only finitely many $c\in k$ for which $f_{d,c}$ is PCF.
For example, for $d=2$ and $k=\QQ$, the only 
PCF parameters are $c=0,-1,-2$. That is, for 
$c\in\QQ$, the map $z\mapsto z^2+c$
is PCF if and only if $c\in\{0,-1,-2\}$.

We set the following notation throughout this paper.

\vspace{0.2cm}

\begin{tabbing}
\hspace{6mm} \= \hspace{15mm} \=  \kill
\> $k$ \> a number field, with algebraic closure $\kbar$ \\
\> $M_k$ \> the standard set of places of $k$ \\
\> $S$ \> a finite subset of $M_k$, including all the archimedean places \\
\> $k_v$ \> the completion of $k$ at a place $v\in M_k$, with absolute value $|\cdot|_v$ \\
\> $\Cv$ \> the completion of an algebraic closure of $k_v$, with absolute value $|\cdot|_v$ \\
\> $f_{d,c}$ \> the polynomial $f_{d,c}(z)=z^d+c$, where $d\geq 2$ is an integer.
\end{tabbing}

\vspace{0.2cm}

%
%

If $L_1$ and $L_2$ are two fields that contain $k$,
we say that a field homomorphism
$\sigma: L_1\to L_2$ is a \emph{$k$-embedding} if it is the identity on $k$.
Such a $k$-embedding $\sigma$ extends to a map from $\PP^1(L_1)$ to $\PP^1(L_2)$
by setting $\sigma(\infty):=\infty$.
If $D$ is an effective divisor on $\PP^1$ defined over $\kbar$,
recall that a point $x\in\Pkbar$ is \emph{$S$-integral} (on $\PP^1$)
relative to $D$, or that $x$ is \emph{$(D,S)$-integral}, if for any place
$v\in M_k$ with $v\not\in S$, for any point $\alpha$ in the support of $D$,
and for any $k$-embeddings $\sigma:\kbar \hookrightarrow \Cv$
and $\tau:\kbar \hookrightarrow \Cv$, the points $\sigma(x)$
and $\tau(\alpha)$ lie in different residue classes of $\PCv$.
In particular, if $x,\alpha\in\kbar$, we have
\[ \begin{cases}
|\sigma(x)-\tau(\alpha)|_v \geq 1 & \text{ if }|\tau(\alpha)|_v \leq 1; \; \;  \text{and}
\\
|\sigma(x)|_v \leq 1 & \text{ if } |\tau(\alpha)|_v > 1.
\end{cases} \]
The above definition is, of course, only a special case --- for $\PP^1$ 
--- of a more general notion of integrality of a point relative to an effective divisor 
on a variety over $k$; see, for example, \cite{GraIh;13}.


For each integer $d\geq 2$, and for each place $v$ of $k$, the family $f_{d,c}(z)=z^d+c$
has an associated $v$-adic \emph{generalized Mandelbrot set} $\mathbf{M}_{d,v}$,
or \emph{multibrot set}, defined by
\begin{equation}
\label{eq:mandel}
\mathbf{M}_{d,v} := \{ c\in\Cv \, : \, \text{the orbit }\{ f_{d,c}^n(0) : n\geq 0\}
\text{ is bounded} \}.
\end{equation}
If $c\in\Cv$ is a PCF parameter for $f_{d,c}$, then clearly $c\in\mathbf{M}_{d,v}$.
If $v$ is an archimedean place, so that $\Cv\cong\CC$, then
$\mathbf{M}_{d,v}$ is the set of parameters $c\in\CC$ for which
the Julia set of $f_{d,c}$ is connected.
It is easy to check that $\mathbf{M}_{d,v}$ is compact for archimedean $v$.

Our main result is as follows.

\begin{thm}
\label{thm:pcffin}
Let $k$ be a number field with algebraic closure $\kbar$,
let $S\subseteq M_k$ be a finite set of places of $k$
including all the archimedean places, let $d\geq 2$ be an integer,
and for any $c\in\kbar$, let $f_{d,c}(z):=z^d+c$.
Let $\alpha\in \kbar$, and suppose that
\begin{itemize}
\item $f_{d,\alpha}$ is not PCF, and
\item for every archimedean place $v$ of $k$,
and for every $k$-embedding $\tau$ of $k(\alpha)$ into $\Cv$,
the image $\tau(\alpha)$ does not lie in the boundary $\partial\mathbf{M}_{d,v}$
of the multibrot set $\mathbf{M}_{d,v}$ of equation~\eqref{eq:mandel}.
\end{itemize}
Then there are only finitely many parameters $c\in\kbar$
that are $S$-integral relative to $(\alpha)$,
and for which $f_{d,c}$ is PCF.
\end{thm}


The hypothesis that  $\alpha\not\in\partial\mathbf{M}_{d.v}$ for 
archimedean $v$ is reminiscent of a 
similar condition called ``totally Fatou" in the context of \cite{Pet;08}.
We conjecture that this condition should not be required, as follows.

\begin{conj}\label{conj:pcffin}
Let $k$ be a number field with algebraic closure $\kbar$,
let $S\subseteq M_k$ be a finite set of places of $k$
including all the archimedean places, let $d\geq 2$ be an integer,
and for any $c\in\kbar$, let $f_{d,c}(z):=z^d+c$.

Let $\alpha\in \kbar$, and suppose that $f_{d,\alpha}$ is not PCF.
Then there are only finitely many parameters $c\in\kbar$
that are $S$-integral relative to $(\alpha)$,
and for which $f_{d,c}$ is PCF.
\end{conj}

On the other hand,
the hypothesis that $\alpha$ is not a PCF parameter cannot be removed,
as we will show in Theorem~\ref{thm:pcfinf}.

For the family $f_{d,c}$, the parameter $c$ lives in a moduli space isomorphic to $\AAA^1$,
and the values of $c$ for which $f_{d,c}$ is PCF may be considered as special points on
this variety, analogous to torsion points on an abelian variety or CM points on a modular curve.
From this perspective, Theorem~\ref{thm:pcffin}, Conjecture~\ref{conj:pcffin},
and Theorem~\ref{thm:pcfinf} describe the integrality of
these (dynamically) special points relative both to special and non-special points
on this moduli space. We propose generalizations of this idea to other moduli spaces,
including higher-dimensional moduli spaces, in Section~\ref{sec:conj}.

Our strategy to prove Theorem~\ref{thm:pcffin} is as follows.
First, at each place $v$ of $k$, we apply the equidistribution theorems
of \cite{GKNY;17} or \cite{Yuan;08}, which say that atomic measures 
supported equally on the Galois orbits of PCF parameters converge weakly
to the so-called \emph{bifurcation measure} of the family $f_{d,c}$ at $v$.
However, we wish to integrate the function
$\log|x-\alpha|_v$ against these measures,
and the discontinuity at $x=\alpha$ means that the equidistribution theorems
do not apply directly in this case.
Therefore, we invoke \cite[Theorem~1.4]{BI;20;1} for 
$v$ non-archimedean, and the hypothesis that $\alpha\not\in\partial\mathbf{M}_{d,v}$
for $v$ archimedean, to prove our desired local convergence result,
which we state as Theorem~\ref{thm:logequi}.

Second,  we write the (strictly) positive canonical height $\hat{h}_{d,\alpha}(\alpha)$
associated with the map $f_{d,\alpha}$ by a sum of canonical local heights.
According to Theorem~\ref{thm:logequi},
given a hypothetical sequence of PCF parameters $(x_n)_{n\geq 1}$ 
in $\kbar$ that are $S$-integral with respect to $(\alpha)$,
we may approximate these canonical local heights by 
integrals of $\log|x-\alpha|_v$ with respect to
atomic measures $\nu_n$ supported on the Galois orbits of $x_n$.
Finally, using the $((\alpha),S)$-integrality of $x_n$, we
rewrite the sum of canonical local heights 
and invoke the product formula to show that the sum approaches $0$,
contradicting the fact that $\hat{h}_{d,\alpha}(\alpha)>0$.

The outline of the paper is as follows.
In Section~\ref{sec:heights}, we recall some fundamental facts about
canonical (local) heights, and we relate the bifurcation measure at $v\in M_k$
of the family $f_{d,c}$ to the canonical local height $\hat{\lambda}_{d,c,v}$.
We then state and prove Theorem~\ref{thm:logequi}
in Section~\ref{sec:logequi}, computing certain canonical local heights
in terms of limits involving PCF parameters.
In Section~\ref{sec:finite}, we use Theorem~\ref{thm:logequi} to
prove Theorem~\ref{thm:pcffin}.
Section~\ref{sec:accumpcf} is devoted to the
statement and proof of Theorem~\ref{thm:pcfinf},
showing that the conclusions of
Theorem~\ref{thm:pcffin} and Conjecture~\ref{conj:pcffin}
fail when $\alpha$ is allowed to be a PCF parameter.
Finally, in Section~\ref{sec:conj}, we state and discuss a 
generalization of Conjecture~\ref{conj:pcffin}.

%

\section{Canonical heights and bifurcation measures}
\label{sec:heights}

\subsection{Call-Silverman canonical heights}
For any place $v\in M_k$ and for $f(z)\in\Cv[z]$ a polynomial of degree $d\geq 2$,
the associated (\emph{Call-Silverman}) \emph{canonical local height function}
$\hat{\lambda}_{f,v}:\Cv\to\RR$ is given by
\begin{equation}
\label{eq:locht}
\hat{\lambda}_{f,v}(x) :=
\lim_{n\to\infty} \frac{1}{d^n} \log\max\big\{1, \big| f^n(x) \big|_v\big\} .
\end{equation}
The function $\hat{\lambda}_{f,v}$ takes nonnegative values, and it is strictly positive
exactly at points $x\in\Cv$ for which $f^n(x)\to\infty$ as $n\to\infty$.
(That is, $\hat{\lambda}_{f,v}$ is zero precisely on the \emph{filled Julia set} of $f$ at $v$,
i.e., on the set of points $z$ that do not escape to $\infty$ under iteration of $f$.)
Moreover, $\hat{\lambda}_{f,v}$ differs from the
standard local height function $\lambda_v(x):=\log\max\{1,|x|_v\}$
by a bounded amount, and the two coincide for all but finitely many $v$.

For a polynomial $f(z)\in k[z]$ of degree $d\geq 2$, 
the associated (\emph{Call-Silverman}) \emph{canonical height function}
$\hat{h}_{f}:k \to\RR$ is  given by
\begin{equation}
\label{eq:globht}
\hat{h}_f(x) =  \frac{1}{ [k : \mathbb Q]} \sum_{v\in M_k} N_v \hat{\lambda}_{f,v}(x)
= \lim_{n\to\infty} \frac{1}{d^n} h\big( f^n(x) \big).
\end{equation}
The coefficients $N_v=[k_v:\Qp]$ in equation~\eqref{eq:globht}, where $v|p$,
are the integers appearing in the product formula over $k$, i.e.,
$\sum_v N_v \log |x|_v= 0$ for all $x\in k^{\times}$.
The function $h$ is the standard Weil height on $k$, given by 
\[ h(x):= \frac{1}{ [k : \mathbb Q]}\sum_{v\in M_k} N_v \log\max\{1,|x|_v\}
= \frac{1}{ [k : \mathbb Q]} \sum_{v\in M_k} N_v \lambda_v(x) . \]
Both $h$ and $\hat{h}_f$ have natural extensions to $\kbar$, but we will only need
their values on $k$, for which the above definitions suffice.
The function $\hat{h}_f$ takes on nonnegative values, 
with $\hat{h}_f(x)=0$ if and only if $x\in\kbar$ is preperiodic under $f$.
In addition, $\hat{h}_f$ differs from $h$ by a bounded amount,
and it satisfies the functional equation $\hat{h}_f(f(z))=d\hat{h}_f(z)$.
Call-Silverman heights were introduced in \cite{CalSil;93};
see also \cite[Sections~3.4--3.5]{Sil;ADS}.

\medskip

\subsection{The bifurcation measure}
To simplify notation, for the polynomial $f_{d,c}(z)=z^d+c$, we will denote
the associated canonical local height function at $v$ by $\hat{\lambda}_{d,c,v}$,
and the associated canonical height function by $\hat{h}_{d,c}$.
If we view the parameter $c$ as the variable in this notation, then we obtain
the Green's function $G_{d,v}:\Cv\to\RR$ given by
\[ G_{d,v}(c):=\hat{\lambda}_{d,c,v}(c) =
\lim_{n\to\infty} \frac{1}{d^n} \log\max\big\{1, \big| f_{d,c}^n(c) \big|_v \big\}. \]
That is, $G_{d,v}$ measures the $v$-adic escape rate of the critical point of $f_{d,c}$,
and hence $G_{d,v}$ is zero precisely on the multibrot set $\mathbf{M}_{d,v}$,
and strictly positive on $\Cv\smallsetminus\mathbf{M}_{d,v}$.

For each place $v\in M_k$, recall that $\PBerkv$
denotes the \emph{Berkovich projective line} at $v$.
If $v$ is an archimedean place, then $\Cv$ may be identified with $\CC$,
and hence $\PBerkv$ is simply the Riemann sphere $\PC$.
On the other hand, if $v$ is non-archimedean, then $\PBerkv$ properly contains $\PCv$.
In particular, for each $a\in\Cv$ and $r>0$, there is a point $\zeta(a,r)\in\PBerkv$
corresponding to the closed disk $\Dbar(a,r):=\{x\in\Cv : |x-a|_v\leq r\}$.
The Berkovich point $\zeta(0,1)$ corresponding to the closed unit disk is called
the \emph{Gauss point}.
For background on $\PBerkv$,
see \cite[Chapter~6]{BenBook} or \cite[Chapters~1--2]{BR;10}.

If $v$ is an archimedean place, so that $\Cv\cong\CC$,
the (potential-theoretic) Laplacian of
$G_{d,v}$ is a probability measure $\mu_{d,v}$ on $\CC$,
called the \emph{bifurcation measure} of the family $f_{d,c}$.
As its name suggests, the support of $\mu_{d,v}$
is precisely the bifurcation locus $\partial\mathbf{M}_{d,v}$ of the family;
see, for example, \cite[Proposition~3.3.(5)]{BD;11} or \cite[Sections~4.1--4.2]{GKNY;17}.

If $v$ is a non-archimedean place, then it is easy to see that
the sequence $(|f_{d,c}^n(c)|_v)_{n\geq 1}$ is bounded if and only if $|c|_v\leq 1$.
In fact, we have the explicit formula $G_{d,v}(c)=\log\max\{1,|c|_v\}$,
which has a unique continuous extension to $\PBerkv\smallsetminus\{\infty\}$.
There is a Laplacian operator on $\PBerkv$; see, for example,
\cite[Section~13.4]{BenBook} for a brief survey, or \cite[Chapters~3--5]{BR;10}
for a detailed exposition. As shown in \cite[Proposition~3.7.(1)]{BD;11},
the Laplacian of $G_{d,v}$, when restricted to
$\PBerkv\smallsetminus\{\infty\}$, is the desired probability measure $\mu_{d,v}$.
In our case,
this measure is $\mu_{d,v}=\delta_{\zeta(0,1)}$, the delta measure
at the Gauss point; see \cite[Example~5.19]{BR;10} or \cite[Example~13.26]{BenBook}.

Since the bifurcation measure $\mu_{d,v}$ is defined as the Laplacian of
$G_{d,v}(c)=\hat{\lambda}_{d,c,v}(c)$, it should not be surprising that we can
recover the canonical local height $\hat{\lambda}_{d,c,v}(c)$ by integrating
an appropriate kernel against $\mu_{d,v}$, as follows.

\begin{lemma}
\label{lem:locht}
Let $d\geq 2$ be an integer, let $v\in M_k$, and let $\alpha \in \Cv$.
Let $\hat{\lambda}_{d,\alpha,v}$ be the canonical local height function of
equation~\eqref{eq:locht} for the map $f_{d,\alpha}(z)=z^d+\alpha$.
Let $\mu_{d,v}$ be the bifurcation measure of the family $f_{d,c}$. Then
\begin{equation}
\label{eq:lochtint}
\int_{\PBerkv} \log |x-\alpha|_v \, d\mu_{d,v}(x)
= \hat{\lambda}_{d,\alpha,v}(\alpha) .
\end{equation}
\end{lemma}

\begin{proof}
Since $\hat{\lambda}_{d,\alpha,v}(\alpha) = G_{d,v}(\alpha)$,
and since $\Delta_v G_{d,v}= \mu_{d,v}$,
where $\Delta_v$ is the Laplacian on $\PBerkv$, the desired
statement is the content of \cite[Example~5.22]{BR;10}.
More precisely, in the notation of equation~(5.8) of that example,
we use $\zeta=\infty$, $\nu=\mu_{d,v}$, and $u_\nu(x,\zeta)=G_{d,v}(x)$,
and we bear in mind that the Laplacian of \cite{BR;10} is the negative of our $\Delta_v$.
\end{proof}

\begin{remark}
If we fix the parameter $c\in\Cv$, then the Laplacian of the canonical local height
function $\hat{\lambda}_{d,c,v}$ is a probability measure $\rho_{d,c,v}$
on $\PBerkv\smallsetminus\{\infty\}$, known as the \emph{equilibrium measure}
of the polynomial $f_{d,c}$ at $v$. Note that the bifurcation measure $\mu_{d,v}$
is a measure on the parameter space (corresponding to $c$),
whereas the equilibrium measure
$\rho_{d,c,v}$ is on the dynamical space (corresponding to $z$).
The support of the equilibrium measure is precisely the ($v$-adic)
\emph{Julia set} of $f_{d,c}$, i.e., the boundary of the filled Julia set.
By similar reasoning as in the proof of Lemma~\ref{lem:locht},
we can also express $\hat{\lambda}_{d,\alpha,v}(\alpha)$ in terms of the equilibrium measure.
Specifically, we can expand equation~\eqref{eq:lochtint} to
\[ \int_{\PBerkv} \log |x-\alpha|_v \, d\mu_{d,v}(x)
= \hat{\lambda}_{d,\alpha,v}(\alpha) 
= \int_{\PBerkv}  \log |x-\alpha|_v\, d\rho_{d,\alpha,v}(x) .\]
\end{remark}

\section{Logarithmic equidistribution of PCF points}
\label{sec:logequi}

\begin{defin}
\label{def:equidist}
Let $v\in M_k$ and let $(x_n)_{n\geq 1}$ be a sequence of points in $\PP^1(\kbar)$.
For each integer $n\geq 1$, let $\nu_n$ be the atomic probability measure
\begin{equation}
\label{eq:atomdef}
\nu_n := \frac{1}{|Gx_n|} \sum_{y\in Gx_n} \delta_y,
\end{equation}
where $Gx_n$ denotes the set of $\Gal(\kbar/k)$-conjugates of $x_n$,
and $\delta_y$ is the delta-measure on $\PBerkv$ supported at $y$.

Let $\mu$ be a Borel probability measure on $X:=\PBerkv$.
We say that the Galois orbits of $(x_n)_{n\geq 1}$ 
are \emph{equidistributed}
with respect to $\mu$ if $(\nu_n)_{n\geq 1}$ converges weakly to $\mu$,
i.e., if for every continuous, compactly supported function
$g:X\to\RR$, we have
\begin{equation}
\label{eq:equidef}
\lim_{n\to\infty} \int_X g(x) \, d\nu_n(x) = \int_X g(x) \, d\mu(x) .
\end{equation}
\end{defin}

For archimedean $v$,
the boundary $\partial\mathbf{M}_{d,v}$ of the multibrot set $\mathbf{M}_{d,v}$
is the bifurcation locus for the family $f_{d,c}$.
Indeed, for any $c\in\partial\mathbf{M}_{d,v}$, there are nearby parameters
$\gamma\in\Cv\cong\CC$
for which $\gamma\in\mathbf{M}_{d,v}$ and hence $f_{d,\gamma}$ has connected Julia set,
and others  for which $\gamma\not\in\mathbf{M}_{d,v}$
and hence $f_{d,\gamma}$ has disconnected Julia set.
All parameters $c\in\CC$ for which $f_{d,c}$ is PCF clearly lie in $\mathbf{M}_{d,v}$.
If $z=0$ is periodic under $f_{d,c}$, then $z=0$ is a (super)attracting periodic
point of $f_{d,\gamma}$.
In that case, the map $f_{d,c}$ is hyperbolic
(see, for example, \cite[Section~V.2]{CG} or \cite[Section 14.1]{DH;84}),
and hence $c$ lies in the interior of $\mathbf{M}_{d,v}$. In fact, since the multiplier of
the (unique) attracting cycle of $f_{d,\gamma}$
varies analytically for $\gamma$ near such $c$,
there is an open neighborhood of $c$ containing no other parameters $\gamma$
for which $f_{d,\gamma}$ is PCF. 
On the other hand, if $z=0$ is strictly preperiodic under $f_{d,c}$, then $c$ is called
a \emph{Misiurewicz parameter}, and we have $c\in\partial\mathbf{M}_{d,v}$.
(See, for example, \cite[Section~VIII.1]{CG}.)

\begin{thm}
\label{thm:logequi}
Let $v\in M_k$ and $\alpha\in \Cv$. Fix an integer $d\geq 2$, and 
let $(x_n)_{n\geq 1}$ be a sequence of distinct points in $\kbar\smallsetminus\{\alpha\}$
such that $z\mapsto z^d+x_n$ is PCF for each $n\geq 1$.
If $v$ is an archimedean place of $k$, assume that $\alpha\not\in\partial\mathbf{M}_{d,v}$.
Then
\begin{equation}
\label{eq:lim}
\lim_{n \rightarrow \infty} \frac{1}{ [ k(x_n) : k ] } 
\sum_{\sigma}
\log | x_n^{\sigma} - \alpha |_{v} =   
\int_{ \PBerkv } \log \big| x - \alpha \big|_{v} \; d\mu_{d, v} (x),
\end{equation}
where the sum is over all $k$-embeddings $\sigma:k(x_n)\hookrightarrow\Cv$.
\end{thm}

We note that if the function $x\mapsto\log |x-\alpha|_v$ were continuous on $\Cv$,
then equation~\eqref{eq:lim} would simply be an instance of equidistribution,
that is, equation~\eqref{eq:equidef} with $g(x)=\log |x-\alpha|_v$. However,
the discontinuity at $x=\alpha$ means that equidistribution results
(like \cite[Theorem~7.52]{BR;10} or \cite[Theorem~3.1]{Yuan;08})
do not apply directly here.

When $v$ is non-archimedean, the function $x\mapsto  \log | x - \alpha |_{v}$
has a unique continuous extension to $\PBerkv\smallsetminus\{\alpha,\infty\}$,
and it is this extension 
that appears as the integrand in equation~\eqref{eq:lim}.
In particular, this extended function maps the Gauss point $\zeta(0,1)$
to $\log\max\{1,|\alpha|_v\}$; intuitively, this is the generic value of $\log|x-\alpha|_v$
for $x$ in the closed unit disk $\Dbar(0,1)$.
As noted just before Lemma~\ref{lem:locht},
we also have $\mu_{d,v}=\delta_{\zeta(0,1)}$ in this case,
and hence the integral in equation~\eqref{eq:lim} evaluates simply to
$\log\max\{1,|\alpha|_v\}$.

As we remarked following Theorem~\ref{thm:pcffin},
we expect that the condition that $\alpha\not\in\partial\mathbf{M}_{d,v}$
for archimedean $v$ should not be required in Theorem~\ref{thm:logequi},
provided $\alpha\in\kbar$.
If so, then the resulting strengthened version of Theorem~\ref{thm:logequi} for such $\alpha$ 
would yield Conjecture~\ref{conj:pcffin}, by the argument given in Section~\ref{sec:finite} below.

\begin{proof}[Proof of Theorem~\ref{thm:logequi}]
\textbf{Case 1}: $v$ is an archimedean place. 
By hypothesis, there exists $r>0$ such that
the open disk $D(\alpha,r):=\{y\in\Cv : |y-\alpha|_v < r \}$ does not intersect the closed set
$\partial\mathbf{M}_{d,v}$. If $\alpha$ lies outside $\mathbf{M}_{d,v}$, then 
$f_{d,\gamma}$ is not PCF for any $\gamma\in D(\alpha,r)$.
On the other hand, if $\alpha\in\mathbf{M}_{d,v}$, then
as noted in the discussion following Definition~\ref{def:equidist},
some neighborhood of $\alpha$ contains no PCF parameters, except perhaps $\alpha$
itself. Either way, then, we may decrease $r>0$ if necessary so that
$f_{d,\gamma}$ is not PCF for any $\gamma\in D(\alpha,r)\smallsetminus\{\alpha\}$.

Let $\psi:\Cv\to [0,1]$ be a continuous function that is $1$ on $\mathbf{M}_{d,v}$
and is $0$ outside some large disk containing $\mathbf{M}_{d,v}$.
Define $g:\Cv\to\RR$ by $g(x):=\psi(x)\cdot\log\max\{r, |x-\alpha|_v\}$,
where $r$ is as defined in the previous parargraph.
Observe that $g$ is continuous and has compact support.
In light of the previous paragraph, since each $x_n^{\sigma}$ is PCF, we have
\begin{equation}
\label{eq:gxn}
g(x_n^{\sigma})=\log |x_n^{\sigma}-\alpha|_v
\end{equation}
for all $n\geq 1$ and all $k$-embeddings  $\sigma:k(x_n)\hookrightarrow\Cv$.

By \cite[Theorem~3.1]{GKNY;17}, the Galois orbits of $(x_n)_{n\geq 1}$
are equidistributed with respect to the bifiurcation measure $\mu_{d,v}$.
Thus, defining the measures $\nu_n$
as in equation~\eqref{eq:atomdef}, we have
\begin{align*}
\lim_{n \rightarrow \infty} \frac{1}{ [ k(x_n) : k ] } 
\sum_{\sigma} \log \big| x_n^{\sigma} - \alpha \big|_{v}
&=  
\lim_{n \rightarrow \infty} \frac{1}{ [ k(x_n) : k ] } 
\sum_{\sigma} g\big(x_n^\sigma\big)
\notag \\
&=  \lim_{n \rightarrow \infty} \int_{\PBerkv} g(x) \, d\nu_n(x)
= \int_{\PBerkv} g(x) \, d\mu_{d,v}(x)
\\
&= \int_{\PBerkv} \log|x-\alpha|_v \, d\mu_{d,v}(x),
\notag
\end{align*}
where the first equality is by equation~\eqref{eq:gxn}, the second is by definition of $\nu_n$,
the third is by equidistribution, and the fourth is because the disk $D(\alpha,r)$
is disjoint from the support $\partial\mathbf{M}_{d,v}$ of $\mu_{d,v}$.

\smallskip

\textbf{Case 2}: $v$ is non-archimedean. By \cite[Theorem~1.4]{BI;20;1}, there
are only finitely many parameters $\gamma\in\Cv$ such that $f_{d,\gamma}$ is PCF
with $|\gamma-\alpha|_v<1/2$. The desired equality follows essentially as
in Case~1, with \cite[Corollary~2.10]{BD;11} showing the requisite equidistribution.

We also provide an alternative proof not using equidistribution,
but still using results from \cite{BI;20;1}, as follows.
As noted just before the start of this proof,
the integral on the right side of
equation~\eqref{eq:lim} is simply $\log\max\{1,|\alpha|_v\}$.
Recall that each PCF parameter for the family $f_{d,c}$ is an algebraic integer,
and therefore $|x_n^{\sigma}|_v\leq 1$ for every $n$ and $\sigma$.
Thus, if $|\alpha|_v>1$, then $|x_n^{\sigma}-\alpha|_v=|\alpha|_v$ for every $n$ and $\sigma$,
and hence both sides of equation~\eqref{eq:lim} equal $\log|\alpha|_v$.

It suffices to show that the left side of equation~\eqref{eq:lim} is zero when $|\alpha|_v\leq 1$.
By \cite[Theorem~1.4]{BI;20;1}, for every $0<r<1$, there are only finitely many
PCF parameters of the family $f_{d,c}$ in the disk $\Dbar(\alpha,r)$.
Therefore, for any $0<r<1$, there is some $N\geq 1$ such that $r<|x_n^{\sigma}-\alpha|_v<1$
for any $n\geq N$ and any $k$-embedding $\sigma$. Thus, we have
\[ 0 \geq \lim_{n \to \infty} \frac{1}{ [k(x_n) : k] }  
\sum_{ \sigma} \log \big| x_n^{\sigma} - \alpha \big|_v \geq \log r \]
for every such $r$. Letting $r\nearrow 1$, the limit is $0$, as desired.
\end{proof}

\section{Finiteness of integral PCF points}
\label{sec:finite}
We are now prepared to prove Theorem~\ref{thm:pcffin},
using Theorem~\ref{thm:logequi}.

\begin{proof}[Proof of Theorem~\ref{thm:pcffin}]
Replacing $k$ by a finite extension if necessary,
we may assume that $\alpha\in k$.
Increasing the finite set $S$ if necessary, we may also assume
that $|\alpha|_v\leq 1$ for every $v\in M_k\smallsetminus S$.

As in Section~\ref{sec:heights},
let $\hat{h}_{d,\alpha}$ be the (Call-Silverman)
canonical height on $\Pkbar$ attached to $f_{d,\alpha} (z) = z^d + \alpha$
(and the divisor (${\infty}$)). For any $\beta\in k$, we have
\[ \hat{h}_{d,\alpha} (\beta) = \frac{1}{ [k : \mathbb Q] }
\sum_{ v \in M_k } N_v \hat{\lambda}_{d,\alpha,v} (\beta), \]
where the integers $N_v=[k_v : \Qp]$ are as in equation~\eqref{eq:globht}.
In addition, because $f_{d,\alpha}$ is not PCF, the critical point $z=0$
is not preperiodic, and hence $\hat{h}_{d,\alpha}(\alpha) = d \hat{h}_{d,\alpha}(0)>0$.

Suppose, towards a contradiction, that $(x_n)_{n\geq 1}$ is a sequence of distinct
elements in $\kbar$ which are $S$-integral relative to $(\alpha)$
and which are PCF parameters for the family $f_{d,c}$.
Then by Lemma~\ref{lem:locht} and Theorem~\ref{thm:logequi}, we have
\begin{align}
\label{eq:mainineq1}
0 < \hat h_{d,\alpha}(\alpha)
&= \frac{1}{ [k : \QQ] } \sum_{ v \in M_k } N_v \hat{\lambda}_{d,\alpha,v} (\alpha)
= \sum_{v \in M_k} \frac{N_v}{ [k : \QQ] } 
\int_{\PBerkv} \log | x - \alpha |_v \, d\mu_{d, v} (x)
\notag \\
&= \sum_{ v \in M_k}   \lim_{n \to \infty}  
\frac{N_v}{ [ k(x_n) : \QQ ] } \sum_{ \sigma} \log \big| x_n^{\sigma} - \alpha \big|_v.
\end{align}
Because each $x_n$ is $S$-integral relative to $(\alpha)$,
and because $|x_n^{\sigma}|_v \leq 1$  and $|\alpha|_v\leq 1$ for every 
$v\in M_k\smallsetminus S$ and every $\sigma$,
we have
\begin{equation}
\label{eq:trivval}
\log \big| x_n^{\sigma} - \alpha \big|_v = 0
\quad \text{for every } n\geq 1, \text{ every } \sigma, \text{ and every } v\in M_k\smallsetminus S.
\end{equation}
In particular, the inner sum in expression~\eqref{eq:mainineq1}
is zero for all $v\in M_k\smallsetminus S$.
Thus, the inequality of~\eqref{eq:mainineq1} becomes
\begin{align*}
0 & < \sum_{ v \in S}   \lim_{n \to \infty}
\frac{N_v}{ [ k(x_n) : \QQ ] } \sum_{ \sigma} \log \big| x_n^{\sigma} - \alpha \big|_v
= \lim_{n \to \infty} \frac{1}{ [ k(x_n) : \QQ ] }  \sum_{ v \in S}
\sum_{ \sigma} N_v \log \big| x_n^{\sigma} - \alpha \big|_v
\\
&= \lim_{n \to \infty} \frac{1}{ [ k(x_n) : \QQ ] } \sum_{ v \in M_k}
\sum_{ \sigma} N_v \log \big| x_n^{\sigma} - \alpha \big|_v
=\lim_{n\to\infty} 0 = 0,
\end{align*}
where the second equality is by equation~\eqref{eq:trivval} again,
and the third is by the product formula for the field $k(x_n)$.
Thus, we have $0<0$, yielding the desired contradiction.
\end{proof}

\section{Accumulation at PCF parameters}
\label{sec:accumpcf}

Theorem~\ref{thm:pcffin} requires that the parameter $\alpha$ not be PCF.
The following result shows that this hypothesis cannot be removed.

\begin{thm}
\label{thm:pcfinf}
Let $k$ be a number field with algebraic closure $\kbar$,
let $S\subseteq M_k$ be a finite set of places of $k$
including all the archimedean places, let $d\geq 2$ be an integer,
and for any $c\in\kbar$, let $f_{d,c}(z):=z^d+c$.
Let $r\geq 1$, and let $\alpha_1,\ldots,\alpha_r\in k$ 
be parameters such that each map $f_{d,\alpha_i}$ is PCF.
Then there are infinitely many parameters $c\in\kbar$
that are $S$-integral relative to the divisor $D:=(\alpha_1) + \cdots + (\alpha_r)$
and for which $f_{d,c}$ is PCF.
\end{thm}

\begin{proof}
For each $i=1,\ldots,r$, there are minimal integers $m_i\geq 0$ and $n_i\geq 1$
such that
\begin{equation}
\label{eq:preper}
f_{d,\alpha_i}^{m_i + n_i} (0) = f_{d,\alpha_i}^{m_i} (0).
\end{equation}
That is, $n_i$ is the minimal period of the periodic cycle that the critical point $z=0$
of $f_{d,\alpha_i}$ ultimately lands on,
and $m_i$ is the length of the strictly preperiodic tail of the forward orbit of $z=0$.

Let $n\geq 1$ be any positive integer different from each of $n_1,\ldots,n_r$.
Define
\[ \Phi_{d,n}(c) := \prod_{\ell|n} f_{d,c}^\ell(0)^{\mu(n/\ell)} \in \ZZ[c], \]
where $\mu$ is the M\"{o}bius function. The degree of $\Phi_{d,n}$
is $\sum_{\ell|n} \mu(n/\ell) d^{\ell-1} > 0$, and hence $\Phi_{d,n}$ has (at least one)
root $\beta\in\kbar$.
By \cite[Theorem~1.1]{HT;15}, the critical point $z=0$ is periodic
with minimal period exactly $n$ under the map $f_{d,\beta}$.
(See also \cite[Section~2]{Buff;18}. The polynomials $\Phi_{d,n}$ are called
\emph{Gleason polynomials}. They have simple roots, which are precisely the parameters
for which the critical point $z=0$ is periodic of minimal period $n$.)
Since there are infinitely many choices of such integers $n$, it suffices to show that
for each such $n$, all the roots $\beta\in\kbar$ of $\Phi_{d,n}$ are $S$-integral 
relative to $D$.

Consider such $n$ and $\beta$. For each $i=1,\ldots,r$,
each place $v\in M_k\smallsetminus S$, and all $k$-embeddings
$\sigma:k(\alpha_i)\hookrightarrow\Cv$ and $\tau:k(\beta)\hookrightarrow\Cv$,
we abuse notation and write $\alpha_i$ for $\sigma(\alpha_i)$
and $\beta$ for $\tau(\beta)$.
Because $v$ is non-archimedean, and because 
the orbit of $z=0$ is preperiodic under both
$f_{d,\alpha_i}(z)=z^d+\alpha_i$ and $f_{d,\beta}(z)=z^d+\beta$,
it follows that $|\alpha_i|_v\leq 1$ and $|\beta|_v \leq 1$,
as noted in the discussion preceding Lemma~\ref{lem:locht}.

Thus, both $f_{d,\alpha_i}$ and $f_{d,\beta}$ have explicit good reduction;
see \cite[Section~4.3]{BenBook}, especially Proposition~4.10.(a).
In particular, by \cite[Proposition~4.19]{BenBook},
for any point $y$ in the closed unit disk $\Dbar(0,1)\subseteq\Cv$, the image
of the open disk $D(y,1)$ under $f_{d,\alpha_i}$ is $D(f_{d,\alpha_i}(y),1)$,
and similarly for $f_{d,\beta}$.
Suppose, towards a contradiction, that $|\alpha_i-\beta|_v < 1$
for some $1 \leq i \leq r$. Then by a simple
induction, it follows that $D(f^m_{d,\alpha_i}(0),1) = D(f^m_{d,\beta}(0),1)$ for every $m\geq 0$.

Because of the critical point at $z=0$, the map
$f^m_{d,\alpha_i}:D(0,1)\to D(f^m_{d,\alpha_i}(0),1)$ is not one-to-one,
and similarly for $f_{d,\beta}$.
Since $f^n_{d,\beta}(0)=0$, we have $f^n_{d,\beta}(D(0,1))=D(0,1)$.
If there were some $1\leq m\leq n-1$ such that $f^m_{d,\beta}(D(0,1))=D(0,1)$,
then by \cite[Theorem~4.18.(b)]{BenBook},
the map $f_{d,\beta}$ would only have one periodic point in $D(0,1)$,
and that point would have period $m$; but this contradicts the fact that $z=0$ has
minimal period $n>m$ under $f_{d,\beta}$. Thus, the $n$ disks
\[ D\big( f^j_{d,\alpha_i}(0), 1\big) = D\big( f^j_{d,\beta}(0), 1\big), \quad j=0,1,\ldots,n-1 \]
are distinct and hence disjoint. By \cite[Theorem~4.18.(b)]{BenBook} again, 
this time applied to $f_{d,\alpha_i}$,
there is a unique periodic cycle of $f_{d,\alpha_i}$ in these disks, and it has minimal period $n$.
However, the point $f^{m_i}_{d,\alpha_i}(0)$ lies in one of these disks
and is periodic of minimal period $n_i\neq n$.
By this contradiction, we must have $|\alpha_i-\beta|_v = 1$
for all $1\leq i\leq r$.
\end{proof}

\begin{remark}
The only fact about the finite set $S$ of places
used in proof of Theorem~\ref{thm:pcfinf} is that
$S$ contains all the archimedean places. In particular,
the proof works just fine even if $S$ consists \emph{only} of the archimedean places.
This is possible because the parameters $\beta$
constructed in the proof are roots of Gleason polynomials, so that the critical point
of $f_{d,\beta}$ is periodic.

One might ask whether it is possible to choose the parameters $\beta\in\kbar$
so that the critical point of $f_{d,\beta}$ is \emph{strictly}
preperiodic, i.e., so that each $\beta$ is a Misiurewicz parameter.
The answer is yes, at least if we assume that the finite set $S$ includes
not only all archimedean places of $k$,
but also all non-archimedean places $v$ dividing the degree $d$.
(Still, as with the Gleason case, the set $S$ can be chosen independent of the divisor $D$.)

To see this,
observe that for each Misiurewicz parameter $\beta$,
there are integers $m\geq 2$ and $n\geq 1$
such that $f_{d,\beta}^m(0)=f_{d,\beta}^{m+n}(0)$
but no simpler orbit relations hold.
Conversely, by a degree-counting argument similar to the one we made
for Gleason polynomials, for any such $m$ and $n$,
there are Misiurewicz parameters $\beta$ satisfying this condition.
Thus, we have $f_{d,\beta}^{m-1}(0)=\zeta f_{d,\beta}^{m+n-1}(0)$
for some $d$-th root of unity $\zeta$ with $\zeta\neq 1$.

With notation as in the proof of Theorem~\ref{thm:pcfinf}, let $m\geq 2$ be an integer
different from each of $m_1,\ldots,m_r$, let $n\geq 1$ be any positive integer,
and let $\beta\in \kbar$ be a (Misiurewicz) parameter satisfying
$f_{d,\beta}^{m-1}(0)=\zeta f_{d,\beta}^{m+n-1}(0)$
for some $d$-th root of unity $\zeta$ with $\zeta\neq 1$.
(Clearly there are infinitely many choices of such $m$ and $n$, and hence
infinitely many choices of such $\beta$.)
As in the proof, if we suppose that $|\alpha_i-\beta|_v < 1$ for some
$v\in M_k\smallsetminus S$ and some $1 \leq i \leq r$,
then $D(f^\ell_{d,\alpha_i}(0),1) = D(f^\ell_{d,\beta}(0),1)$ for every $\ell\geq 0$.

Let $R=|f_{d,\beta}^{m-1}(0)|_v$; then $R>0$ because $\beta$ is Misiurewicz.
Since $v\nmid d$, we have $|\zeta-1|_v=1$,
and hence the three points $0$, $f_{d,\beta}^{m-1}(0)$,
and $f_{d,\beta}^{m+n-1}(0)=\zeta^{-1}f_{d,\beta}^{m-1}(0)$
each have $v$-adic distance $R$ from one another. If $R<1$, then
the disk $D(0,1)$ would map into itself under $f_{d,\beta}^n$, and since this map is
not one-to-one, the $n$-periodic point $f_{d,\beta}^{m+n-1}(0)$ would be attracting.
Moreover, because of the attracting periodic point, $D(0,1)$ maps into itself under $f_{d,\beta}^n$
but under no smaller iterate of $f_{d,\beta}$. Since $D(0,1)$ also maps into itself under $m-1$ iterations,
we must therefore have $n|(m-1)$.
But then, again because of the attracting periodic point,
the distance between $0$ and $f_{d,\beta}^{m+n-1}(0)$ would be strictly
greater than the distance between $f_{d,\beta}^{m-1}(0)$ and $f_{d,\beta}^{m+n-1}(0)$,
a contradiction. Hence, we must have $R=1$.

Recall that $m\neq m_i$.
We may assume without loss that $m>m_i$.
Indeed, if $m_i>m$, then we may reverse the roles of $\alpha_i$
and $\beta$ in what follows, because in that case, the PCF parameter $\alpha_i$
must be Misiurewicz rather than Gleason.

Define
\[ U:=D\big(f_{d,\alpha_i}^{m-1}(0),1\big)=D\big(f_{d,\beta}^{m-1}(0),1\big) \]
and
\[ V:=D\big(f_{d,\alpha_i}^{m+n-1}(0),1\big)=D\big(f_{d,\beta}^{m+n-1}(0),1\big). \]
Then the disks $U$ and $V$ are distinct, since $R=1$.
On the other hand, because $f_{d,\beta}^{m}(0)=f_{d,\beta}^{m+n}(0)$,
both $U$ and $V$ map to the disk
\[ W:=D\big(f_{d,\alpha_i}^{m}(0),1\big)=D\big(f_{d,\beta}^{m}(0),1\big) \]
under $f_{d,\beta}$ and hence also under $f_{d,\alpha_i}$.
Because the point $f_{d,\alpha_i}^{m_i}(0)$ is periodic under $f_{d,\alpha_i}$,
it follows that $f_{d,\alpha_i}^{m-1}(0)$ is also periodic, since $m_i<m$.
Therefore, the distinct disks $U$ and $V$
are both part of a periodic cycle of disks under $f_{d,\alpha_i}$.
However, two distinct elements of a periodic cycle cannot have
the same image, but $U$ and $V$ both map to $W$ under $f_{d,\alpha_i}$.
By this contradiction, we see that $|\alpha_i-\beta|_v = 1$, as desired.

For more on Gleason and Misiurewicz polynomials, see, for example, \cite{Buff;18,HT;15}.
\end{remark}

\section{Conjectural generalizations}
\label{sec:conj}

Fix $d\geq 2$, and let $\Rat_d$ denote the space of rational 
functions $f:\PP^1\to\PP^1$
of degree $d$, which is an affine variety naturally identified with a Zariski open subset
of $\PP^{2d+1}$. 
Following \cite[Section~4.4]{Sil;ADS}, the moduli space $\calM_d$ is the quotient space
of $\Rat_d$ by the conjugation action of $\PGL_2$.
We recall the following definition from \cite[Section~1.4]{BD;13}.

\begin{defin}
\label{def:algfam}
Let $k$ be a number field with algebraic closure $\kbar$.
An \emph{algebraic family of critically marked rational maps of degree $d$} over $k$
is a quasiprojective variety $X$ equipped with
\begin{itemize}
\item a regular map $x\mapsto f_x$ from $X$ to $\Rat_d$, and
\item for each $i=1,\ldots,2d-2$, a regular map $c_i:X\to\PP^1$,
\end{itemize}
all defined over $k$, such that for each $x\in X(\kbar)$, the critical points of $f_x$,
listed with multiplicity, are $c_1(x),\ldots,c_{2d-2}(x)$.
If the image of $X$ under the composition
$X\to\Rat_d \twoheadrightarrow \calM_d$ has dimension $N\geq 0$,
we say that $X$ is an
\emph{$N$-dimensional algebraic family of critically marked rational maps of degree $d$}.
\end{defin}

Given an algebraic family $(X,f_x,c_1,\ldots,c_{2d-2})$ as in Definition~\ref{def:algfam},
let $K:=\kbar(X)$ be the function field of $X$.
Then the family defines a
rational function $\mathbf{f}(z)\in K(z)$ of degree $d$, with critical points $c_i\in\PP^1(K)$, for $i=1,\ldots,2d-2$.
Still following \cite{BD;13}, along with \cite[Section~6]{Dem;16},
for $n\geq 1$, we say an $n$-tuple $(c_{i_1},\ldots,c_{i_n})$ of these marked
critical points is \emph{dynamically dependent} if there is a (possibly reducible)
closed subvariety $Y$ of $(\PP^1)^n$ defined over $K$ such that
\begin{itemize}
\item $(c_{i_1},\ldots,c_{i_n})$ lies on $Y$,
\item $\mathbf{F}(Y)\subseteq Y$, where $\mathbf{F}:=(\mathbf{f},\ldots,\mathbf{f})$,
\item There is some $j\in\{1,\ldots,n\}$ and a nonempty Zariski open subset $X'\subseteq X$
with the following property. Consider the projection map $\pi_j:(\PP^1)^n\to (\PP^1)^{n-1}$
that deletes the $j$-th coordinate. Then for all $x\in X'$,
the restriction of $\pi_j$ to the specialization $Y_x$ is finite.
\end{itemize}
(In \cite{Dem;16}, DeMarco calls such a subvariety $Y$ a \emph{dynamical relation},
and adds the third condition, which did not appear in \cite{BD;13},
in order to disallow families with certain degenerations.)
If there is no such dynamical dependence, then we say that the $n$ critical points
$c_{i_1},\ldots,c_{i_n}$ are \emph{dynamically independent}.

For example, if $c_1$ is persistently preperiodic, meaning that there are integers
$s>r\geq 0$ so that $\mathbf{f}^r(c_1)=\mathbf{f}^s(c_1)$,
then the one-tuple $(c_1)$ is dynamically related, with the subvariety $Y\subseteq\PP^1$ consisting of the
(finitely many) points in the forward orbit of $c_1$ under $\mathbf{f}$.
Similarly, if $c_1,c_2$ satisfy $\mathbf{f}^r(c_1)=\mathbf{f}^s(c_2)$ for some $r,s\geq 0$,
then the pair $(c_1,c_2)$ is dynamically related,
with $Y\subseteq\PP^1\times\PP^1$ defined by the equation $\mathbf{f}^r(y_1)=\mathbf{f}^s(y_2)$.


Inspired by Baker and DeMarco's conjecture on Zariski density of PCF points
in dynamical moduli spaces \cite[Conjecture~1.10]{BD;13},
as well as by DeMarco's related conjecture in \cite[Conjecture~6.1]{Dem;16},
we propose the following generalization of our earlier Conjecture~\ref{conj:pcffin}.

\begin{conj}
\label{conj:pcfnondense}
Let $k$ be a number field, let $S$ be a finite set of places of $k$ including
all the archimedean places, and let $d\geq 2$ and $N\geq 1$.
Let $(X,f_x)$ be an $N$-dimensional algebraic family
of critically marked rational maps of degree $d$, defined over $k$.
Let $D$ be a nonzero effective divisor on $X$. Suppose that:
\begin{itemize}
\item the composition $X\to\Rat_d \twoheadrightarrow\calM_d$ is quasifinite,
\item $X$ has at most $N$ dynamically independent critical points, and
\item at least one irreducible component of $D$
has at least $N$ dynamically independent critical points.
\end{itemize}
Then the set
\begin{equation}
\label{eq:intPCF}
\{ x \in X ( \overline k ) : \textup{$f_x$ is PCF, and $x$ is $S$-integral on $X$ relative to $D$} \}
\end{equation}
is not Zariski dense in $X$.
\end{conj}


For example, our family $f_{d,c}(z)=z^d+c$ has parameter $c$ lying in $X=\AAA^1$,
and the largest possible set of dynamically independent critical points has cardinality $1$,
consisting only of the critical point at $z=0$. If $D=(\alpha_1) + \cdots + (\alpha_r)$
with $f_{d,\alpha_1}$ not PCF, then the irreducible component $(c=\alpha_1)$
has a dynamically independent critical point, since the critical point $z=0$ is not preperiodic
under $f_{d,\alpha_1}$. Thus, Conjecture~\ref{conj:pcffin}
implies Conjecture~\ref{conj:pcfnondense} 
for this case, and the two are precisely the same if $D=(\alpha_1)$.

On the other hand, in light of Theorem \ref{thm:pcfinf}, the set of
equation~\eqref{eq:intPCF} for the family $f_{d,c}$
\emph{is} Zariski dense in $X=\AAA^1$ when $D=(\alpha)$ and when $f_{d,\alpha}$ is PCF.
More generally, in the notation of Conjecture~\ref{conj:pcfnondense},
if no irreducible component of $D$ has $N$ (or more) dynamically independent critical points,
then we expect, possibly after enlarging the finite set $S$,
that the set of equation~\eqref{eq:intPCF} is Zariski dense in $X$.


We note that the hypotheses of Conjecture~\ref{conj:pcfnondense} exclude
the case that $X$ is the flexible Latt\`{e}s locus in $\calM_d$ (when $d$ is a square),
since in that case $N=1$, but any component of any divisor $D$ of $X$ would correspond
to a Latt\`{e}s and hence PCF map, and therefore would have
no dynamically independent critical points.


\begin{remark}
As in \cite{BD;13}, the dimension of the \emph{variety} $X$ in Definition~\ref{def:algfam}
might be strictly larger than the dimension $N$ of the \emph{family}.
However, in practice, the map $X\to\calM_d$ is usually quasifinite, which implies that
both $X$ and its image in $\calM_d$ have the same dimension.
For example, our family $c\mapsto f_{d,c}(z)=z^d+c$ has $X=\AAA^1$, and the image in $\calM_d$
also has dimension $1$, since $z^d+a$ is conjugate to $z^d+b$ if and only if $b=\zeta a$
for some $(d-1)$-st root of unity $\zeta$.
Thus, the hypothesis in Conjecture~\ref{conj:pcfnondense} that $X\to\calM_d$ is quasifinite,
which we assume so that both $X$ and the divisor $D$ behave well under this map,
already applies to almost all families of interest.
\end{remark}

For any integer $d\geq 2$, one may define the moduli space $\overline{\calM}_d^1[\calP]$
of critically marked rational functions of degree $d$, up to conjugation, as a geometric quotient scheme;
see \cite[Section~10.1]{DS;18}. That is, each point of $\overline{\calM}_d^1[\calP]$
corresponds to a conjugacy class of tuples $(f,c_1,\ldots,c_{2d-2})$, where $f$ is a rational function
of degree $d$ whose critical points in $\PP^1$ are $c_1,\ldots,c_{2d-2}$.
(In the terminology of \cite[Sections~9--10]{DS;18}, the critical portrait $\calP$ here consists solely of the $2d-2$
marked critical points, but with no restrictions on their orbits. We use $\overline{\calM}_d^1[\calP]$
instead of $\calM_d^1[\calP]$ to allow two or more critical points to coincide in a higher-multiplicity
critical point, while still ensuring that $f$ does not degenerate to a map of lower degree.)
The following conjecture is essentially Conjecture~\ref{conj:pcfnondense}
applied to this geometric moduli space.

\begin{conj}
\label{conj:geompcfnondense}
Let $k$ be a number field, let $S$ be a finite set of places of $k$ including
all the archimedean places, and let $d\geq 2$.
Let $X$ be a closed subvariety of $\overline{\calM}_d^1[\calP]$, defined over $k$,
which has at most $\dim X$ dynamically independent critical points.
Suppose that $D$ is a nonzero effective divisor on $X$
at least one of whose irreducible components has at least $\dim X$ 
dynamically independent critical points.
Then the set
\begin{equation}
\label{eq:intPCF2}
\{ x \in X ( \overline k ) : \textup{$x$ is PCF and $S$-integral on $X$ relative to $D$} \}
\end{equation}
is not Zariski dense in $X$, where $x$ being PCF means that $x$ corresponds to
a PCF rational map.
\end{conj}

Conversely, again in light of Theorem~\ref{thm:pcfinf}, if no irreducible component of the divisor $D$
in Conjecture~\ref{conj:geompcfnondense} has $\dim X$ dynamically independent
critical points, then we expect, possibly after enlarging the finite set $S$,
that the set~\eqref{eq:intPCF2} \emph{is} Zariski dense in $X$.

Returning to the family $f_{d,c}(z)=z^d+c$, we also propose the following
further integrality conjecture, inspired by the special case of the Dynamical
Andr\'{e}-Oort Conjecture proven in \cite[Theorem 1.1]{GKNY;17}.

\begin{conj}
\label{conj:a2}
Let $k$ be a number field, let $S$ be a finite set of places of $k$ including
all the archimedean places, and let $d\geq 2$ be an integer.
Write $f_{d,c}(z):=z^d+c$.
Let $D$ be a nonzero effective divisor on $\AAA^2$ such that at least one of
its irreducible components is \emph{not} of any of the following three forms:
\begin{enumerate}
\item $\{ c \} \times \AAA^1$, where $c\in\kbar$ and $f_{d,c}$ is PCF,
\item $ \AAA^1 \times \{ c \}$, where $c\in\kbar$ and $f_{d,c}$ is PCF,
\item the solution set of $x - \zeta y = 0$, where $\zeta$ is a $(d-1)$st root of unity.
\end{enumerate}
Then the set
\begin{equation}
\label{eq:a2PCF}
\bigg\{ P=(a,b) \in\AAA^2(\kbar) \bigg|
\begin{array}{l}
P \textup{ is } S\textup{-integral on } \AAA^2 \textup{ relative to } D,
\\
\hfil \textup{ and both } f_{d,a} \textup{ and} f_{d,b} \textup{ are PCF} \hfil
{}
\end{array}
\bigg\}
\end{equation}
is not Zariski dense in $\AAA^2$.
\end{conj}

Conversely, once again in analogy with Theorem~\ref{thm:pcfinf}, if every irreducible component
of the divisor $D$ in Conjecture~\ref{conj:a2} is of one of the three forms (a)--(c),
then we expect that the set~\eqref{eq:a2PCF} \emph{is} Zariski dense in $\AAA^2$,
possibly after enlarging the finite set $S$.

All of the preceding conjectures, and not just Conjecture~\ref{conj:a2},
may be viewed as integrality variants of the Dynamical Andr\'{e}-Oort Conjecture
described in \cite{BD;11,BD;13,GKNY;17}, wherein PCF points play the role
of special points. We therefore close with the following conjecture,
in the Shimura variety setting of the original Andr\'{e}-Oort Conjecture.
For the notion of \emph{special} points or subvarieties of a Shimura variety,
we refer the reader to \cite[Section 1]{UY;14}.

\begin{conj} \label{shim} 
Let $k$ be a number field, let $S$ be a finite set of places of $k$ including all the archimedean places,
and let $X$ be a special subvariety of a Shimura variety, defined over $k$.
If $D$ is a nonzero effective divisor on $X$ at least one of whose irreducible
components is not special, then the set 
\[ 
\{ P \in X ( \kbar ) : \textup{$P$ is special and $S$-integral on $X$ relative to $D$} \}
\]
is not Zariski dense in $X$.
\end{conj}

\medskip

\noindent
\textbf{Acknowledgments}.
The first author gratefully acknowledges the support of NSF grant DMS-150176.
The second author gratefully acknowledges the support of Simons Foundation grant 
622375 and the hospitality of the Korea Institute for Advanced Study during his visit.
The authors thank Laura DeMarco for helpful discussions.
We also express our gratitude to the mathematical legacy of Lucien Szpiro,
whose deep insights have influenced the underlying philosophy of this article;
we dedicate this paper to his memory.


\end{document}